\documentclass{gtpart}
\usepackage{pinlabel}
\usepackage{graphicx}
\title[Note on residual finiteness of Artin groups]{Note on residual finiteness of Artin groups}
\author[R Blasco-Garc\'ia]{Rub\'en Blasco-Garc\'ia}
\givenname{Rub\'en}
\surname{Blasco}
\address{Departamento de Matem\'aticas\\
Universidad de Zaragoza\\
50009 Zaragoza\\
Spain}
\email{rubenb@unizar.es}
\urladdr{}
\author[A Juh\'asz]{Arye Juh\'asz}
\givenname{Arye}
\surname{Juh\'asz}
\address{Department of Mathematics\\
Technion, Israel Institute of Technology\\
Haifa 32000\\
Israel}
\email{arju@technion.ac.il}
\urladdr{}
\author[L Paris]{Luis Paris}
\givenname{Luis}
\surname{Paris}
\address{IMB, UMR 5584\\
CNRS, Univ. Bourgogne Franche-Comt\'e\\
21000 Dijon\\
France}
\email{lparis@u-bourgogne.fr}
\urladdr{}


\volumenumber{}
\issuenumber{}
\publicationyear{}
\papernumber{}
\startpage{}
\endpage{}
\doi{}
\MR{}
\Zbl{}
\received{}
\revised{}
\accepted{}
\published{}
\publishedonline{}
\proposed{}
\seconded{}
\corresponding{}
\editor{}
\version{}
\newtheorem{thm}{Theorem}[section]
\newtheorem{lem}[thm]{Lemma}

\newtheorem{corl}[thm]{Corollary}

\theoremstyle{definition}
\newtheorem*{rem}{Remark}

\numberwithin{equation}{section}

\makeatletter
\renewcommand{\thefigure}{\ifnum \c@section>\z@ \thesection.\fi
 \@arabic\c@figure}
\@addtoreset{figure}{section}
\makeatother

\begin{document}

\def\N{\mathbb N} \def\PP{\mathcal P} \def\Z{\mathbb Z}


\begin{abstract}
Let $A$ be an Artin group.
A partition $\PP$ of the set of standard generators of $A$ is called admissible if, for all $X,Y \in \PP$, $ X \neq Y$, there is at most one pair $(s,t) \in X \times Y$ which has a relation.
An admissible partition $\PP$ determines a quotient Coxeter graph $\Gamma/\PP$.
We prove that, if $\Gamma/\PP$ is either a forest or an even triangle free Coxeter graph and $A_X$ is residually finite for all $X \in \PP$, then $A$ is residually finite.
\end{abstract}

\maketitle


\section{Introduction and statements}

Let $S$ be a finite set.
A \emph{Coxeter matrix} over $S$ is a square matrix $M = (m_{s,t})_{s,t \in S}$, indexed by the elements of $S$, with coefficients in $\N \cup \{ \infty\}$, such that $m_{s,s}=1$ for all $s \in S$ and $m_{s,t} = m_{t,s} \ge 2$ for all $s,t \in S$, $s \neq t$.
We represent $M$ by a labelled graph $\Gamma$, called \emph{Coxeter graph}, defined as follows. 
The set of vertices of $\Gamma$ is $S$ and two vertices $s,t$ are connected by an edge labelled with $m_{s,t}$  if $m_{s,t} \neq \infty$.

If $a,b$ are two letters and $m$ is an integer $\ge 2$, we set $\Pi(a,b:m) = (ab)^{\frac{m}{2}}$ if $m$ is even, and $\Pi(a,b:m) = (ab)^{\frac{m-1}{2}}a$ if $m$ is odd. 
In other words, $\Pi(a,b:m)$ denotes the word $aba \cdots$ of length $m$.
The \emph{Artin group} $A=A_\Gamma$ of $\Gamma$ is defined by the presentation
\[
A = \langle S \mid \Pi(s,t: m_{s,t}) = \Pi(t,s: m_{s,t}) \text{ for all } s,t \in S,\ s \neq t \text{ and } m_{s,t} \neq \infty \rangle\,.
\]

Recall that a group $G$ is \emph{residually finite} if for each $g \in G \setminus \{1\}$ there exists a group homomorphism $\varphi: G \to K$ such that $K$ is finite and $\varphi (g) \neq 1$.
Curiously the list of Artin groups known to be residually finite is quite short.
It contains the Artin groups of spherical type (because they are linear), the right-angled Artin groups (because they are residually torsion free nilpotent), the even Artin groups of FC type (see Blasco-Garcia--Martinez-Perez--Paris \cite{BlMaPa1}) and some other examples.
Our purpose here is to increase this list.

For $X \subset S$ we denote by $\Gamma_X$ the full subgraph of $\Gamma$ spanned by $X$ and by $A_X$ the subgroup of $A$ generated by $X$. 
By Van der Lek \cite{Lek1}, $A_X$ is the Artin group of $\Gamma_X$.
We say that $\Gamma$ is \emph{even} if $m_{s,t}$ is either even or $\infty$ for all $s,t \in S$, $s \neq t$.
We say that $\Gamma$ is \emph{triangle free} is $\Gamma$ has no full subgraph which is a triangle. 
Here by a \emph{partition} of $S$ we mean a set $\PP$ of pairwise disjoint subsets of $S$ satisfying $\cup_{X \in \PP} X = S$.
We say that a partition $\PP$ is \emph{admissible} if, for all $X, Y \in \PP$, $X \neq Y$, there is at most one edge in $\Gamma$ connecting an element of $X$ with an element of $Y$.
In particular, if $s \in X$ and $t \in Y$ are connected in $\Gamma$ by an edge and $s' \in X$, $s' \neq s$, then $s'$ is not connected in $\Gamma$ by an edge to any vertex of $Y$.  
An admissible partition $\PP$ of $\Gamma$ determines a new Coxeter graph $\Gamma/\PP$ defined as follows. 
The set of vertices of $\Gamma/\PP$ is $\PP$.
Two distinct elements $X,Y \in \PP$ are connected by an edge labelled with $m$ if there exist $s \in X$ and $t \in Y$ such that $m_{s,t}=m$ ($\neq \infty$).

Our main result is the following.

\begin{thm}\label{thm1_1}
Let $\Gamma$ be a Coxeter graph, let $A = A_\Gamma$, and let $\PP$ be an admissible partition of $S$ such that
\begin{itemize}
\item[(a)]
the group $A_X$ is residually finite for all $X \in \PP$,
\item[(b)]
the Coxeter graph $\Gamma/\PP$ is either even and triangle free, or a forest.
\end{itemize}
Then $A$ is residually finite.
\end{thm}

\begin{corl}\label{corl1_2}
\begin{itemize}
\item[(1)]
If $\Gamma$ is even and triangle free, then $A$ is residually finite.
\item[(2)]
If $\Gamma$ is a forest, then $A$ is residually finite.
\end{itemize}
\end{corl}

\begin{rem}
Recall that the \emph{Coxeter group} $W$ of $\Gamma$ is the quotient of $A$ by the relations $s^2=1$, $s \in S$.
We say that $\Gamma$ is of \emph{spherical type} if $W$ is finite.
We say that a subset $X$ of $S$ is \emph{free of infinity} if $ m_{s,t} \neq \infty$ for all $s,t \in X$.
We say that $\Gamma$ is of \emph{FC type} if for each free of infinity subset $X$ of $S$ the Coxeter graph $\Gamma_X$ is of spherical type.
Artin groups of FC type were introduced by Charney--Davis \cite{ChaDav1} in their study of the $K (\pi,1)$ problem for Artin groups and there is an extensive literature on them.
It is easily checked that any triangle free Coxeter graph (in particular any forest) is of FC type.
So, all the groups that appear in Corollary \ref{corl1_2} are of FC type.
On the other hand, we know by Blasco-Garcia--Martinez-Perez--Paris \cite{BlMaPa1} that all even Artin groups of FC type are residually finite (which gives an alternative proof to Corollary \ref{corl1_2}\,(1)).
The next challenge would be to prove that all Artin groups of FC type are residually finite.
Another interesting challenge would be to prove that all three generators Artin groups are residually finite.
\end{rem}

{\bf Acknowledgement}
The first named author was partially supported by Gobierno de Arag\'on, European Regional Development Funds, MTM2015-67781-P (MINECO/ FEDER) and by the Departamento de Industria e Innovaci\'on del Gobierno de Arag\'on and Fondo Social Europeo Phd grant.


\section{Proof of Theorem \ref{thm1_1}}

The proof of Theorem \ref{thm1_1} is based on the following. 

\begin{thm}[Boler--Evans \cite{BolEva1}]\label{thm2_1}
Let $G_1$ and $G_2$ be two residually finite groups and let $L$ be a common subgroup of $G_1$ and $G_2$.
Assume that both $G_1$ and $G_2$ split as semi-direct products $G_1 = H_1 \rtimes L$ and $G_2 = H_2 \rtimes L$.
Then $G= G_1 \ast_L G_2$ is residually finite.
\end{thm}

The rest of the section forms the proof of Theorem \ref{thm1_1}.

\begin{lem}\label{lem2_2}
If a Coxeter graph $\Gamma$ has one or two vertices, then $A=A_\Gamma$ is residually finite.
\end{lem}

\begin{proof}
If $\Gamma$ has only one vertex, then $A \simeq \Z$ which is residually finite.
Suppose that $\Gamma$ has two vertices $s,t$.
If $m_{s,t} = \infty$, then $A$ is a free group of rank $2$ which is residually finite.
If $m_{s,t} \neq \infty$, then $\Gamma$ is of spherical type, hence, by Digne \cite{Digne1} and Cohen--Wales \cite{CohWal1}, $A$ is linear, and therefore $A$ is residually finite.
\end{proof}

\begin{lem}\label{lem2_3}
Let $\Gamma$ be a Coxeter graph and let $A = A_\Gamma$.
Let $s \in S$. 
We set $Y= S \setminus \{s\}$, we denote by $\Gamma_1, \dots, \Gamma_\ell$ the connected components of $\Gamma_Y$ and, for $i \in \{1, \dots, \ell\}$, we denote by $Y_i$ the set of vertices of $\Gamma_i$.
If $A_{Y_i \cup \{ s\}}$ is residually finite for all $i \in \{1, \dots, \ell\}$, then $A$ is residually finite.
\end{lem}

\begin{proof}
We argue by induction on $\ell$.
If $\ell = 1$, then $Y_1 \cup \{s\} = S$ and $A_{Y_1 \cup \{s\}} = A$, hence $A$ is obviously residually finite. 
Suppose that $\ell \ge 2$ plus the inductive hypothesis. 
We set $X_1 = Y_1 \cup \cdots \cup Y_{\ell-1} \cup \{s\}$, $X_2 = Y_\ell \cup \{s\}$ and $X_0 = \{s\}$.
Let $G_1 = A_{X_1}$, $G_2 = A_{X_2}$, and $L= A_{X_0} \simeq \Z$.
The group $G_1$ is residually finite by induction and $G_2$ is residually finite by the starting hypothesis. 
It is easily seen in the presentation of $A$ that $A = G_1 \ast_L G_2$.
Furthermore, the homomorphism $\rho_1 : G_1 \to L$ which sends $t$ to $s$ for all $t \in X_1$ is a retraction of the inclusion map $L \hookrightarrow G_1$, hence $G_1$ splits as a semi-direct product $G_1 = H_1 \rtimes L$.
Similarly, $G_2$ splits as a semi-direct product $G_2 = H_2 \rtimes L$.
We conclude by Theorem \ref{thm2_1} that $A$ is residually finite. 
\end{proof}

\begin{lem}\label{lem2_4}
Let $\Gamma$ be a Coxeter graph, let $A = A_\Gamma$, and let $\PP$ be an admissible partition of $S$ such that
\begin{itemize}
\item[(a)]
the group $A_X$ is residually finite for all $X \in \PP$,
\item[(b)]
the Coxeter graph $\Gamma/\PP$ has at most two vertices.
\end{itemize}
Then $A$ is residually finite.
\end{lem}

\begin{proof}
If $|\PP| = 1$ there is nothing to prove. 
Suppose that $|\PP|=2$ and one of the elements of $\PP$ is a singleton.
We set $\PP = \{ X,Y\}$ where $X=S \setminus \{t\}$ and $Y = \{t\}$ for some $t \in S$.
If there is no edge in $\Gamma$ connecting $t$ to an element of $X$, then $A = A_X * A_Y$, hence $A$ is residually finite.
So, we can assume that there is an edge connecting $t$ to an element $s \in X$.
Note that this edge should be unique. 
We denote by $\Gamma_1, \dots, \Gamma_\ell$ the connected components of $\Gamma_{X \setminus \{s\}}$ and, for $i \in \{1, \dots, \ell\}$, we denote by $X_i$ the set of vertices of $\Gamma_i$.
For all $i \in \{1, \dots, \ell\}$ the group $A_{X_i \cup \{s\}}$ is residually finite since $A_{X_i \cup \{s\}} \subset A_X$.
On the other hand, $A_{\{s,t\}}$ is residually finite by Lemma \ref{lem2_2}.
It follows by Lemma \ref{lem2_3} that $A$ is residually finite.

Now assume that $|\PP|=2$ and both elements of $\PP$ are of cardinality $\ge 2$.
Set $\PP = \{ X,Y\}$. 
If there is no edge in $\Gamma$ connecting an element of $X$ with an element of $Y$, then $A = A_X * A_Y$, hence $A$ is residually finite. 
So, we can assume that there is an edge connecting an element $s \in X$ to an element $t \in Y$. 
Again, this edge is unique.
Let $\Omega_1, \dots, \Omega_p$ be the connected components of $\Gamma_{X \setminus \{s\}}$ and let $\Gamma_1, \dots, \Gamma_q$ be the connected components of $\Gamma_Y$.
We denote by $X_i$ the set of vertices of $\Omega_i$ for all $i \in \{1, \dots, p\}$ and by $Y_j$ the set of vertices of $\Gamma_j$ for all $j \in \{1, \dots, q\}$.
The group $A_{X_i \cup \{s\}}$ is residually finite since $X_i \cup \{s \} \subset X$ for all $i \in \{1, \dots, p\}$, and, by the above, the group $A_{Y_j \cup \{s\}}$ is residually finite for all $j \in \{1, \dots, q\}$.
It follows by Lemma \ref{lem2_3} that $A$ is residually finite.
\end{proof}

\begin{rem}
Alternative arguments from Pride \cite{Pride1} and/or from Burillo--Martino \cite{BurMar1} can be used to prove partially or completely Lemma \ref{lem2_4}.
\end{rem}

\begin{lem}\label{lem2_5}
Let $\Gamma$ be a Coxeter graph, let $A = A_\Gamma$, and let $\PP$ be an admissible partition of $S$ such that
\begin{itemize}
\item[(a)]
the group $A_X$ is residually finite for all $X \in \PP$,
\item[(b)]
the Coxeter graph $\Gamma/\PP$ is even and triangle free.
\end{itemize}
Then $A$ is residually finite.
\end{lem}

\begin{proof}
We argue by induction on the cardinality $|\PP|$ of $\PP$.
The case $|\PP| \le 2$ is covered by Lemma \ref{lem2_4}.
So, we can suppose that $|\PP| \ge 3$ plus the inductive hypothesis. 
Since $\Gamma/\PP$ is triangle free, there exist $X,Y \in \PP$ such that none of the elements of $X$ is connected to an element of $Y$.
We set $U_1 = S \setminus X$, $U_2 = S \setminus Y$, and $U_0 = S \setminus (X \cup Y)$.
We have $A = A_{U_1} *_{A_{U_0}} A_{U_2}$ and, by the inductive hypothesis, $A_{U_1}$ and $A_{U_2}$ are residually finite.
Since $\Gamma/\PP$ is even, the inclusion map $A_{U_0} \hookrightarrow A_{U_1}$ admits a retraction $\rho_1 : A_{U_1} \to A_{U_0}$ which sends $t$ to $1$ if $t \in Y$ and sends $t$ to $t$ if $t \in U_0$.
Similarly, the inclusion map $A_{U_0} \hookrightarrow A_{U_2}$ admits a retraction $\rho_2: A_{U_2} \to A_{U_0}$.
By Theorem \ref{thm2_1} it follows that $A$ is residually finite. 
\end{proof}

The following lemma ends the proof of Theorem \ref{thm1_1}.

\begin{lem}\label{lem2_6}
Let $\Gamma$ be a Coxeter graph, let $A = A_\Gamma$, and let $\PP$ be an admissible partition of $S$ such that
\begin{itemize}
\item[(a)]
the group $A_X$ is residually finite for all $X \in \PP$,
\item[(b)]
the Coxeter graph $\Gamma/\PP$ is a forest.
\end{itemize}
Then $A$ is residually finite.
\end{lem}

\begin{proof}
We argue by induction on $|\PP|$.
The case $|\PP| \le 2$ being proved in Lemma \ref{lem2_4}, we can assume that $|\PP| \ge 3$ plus the inductive hypothesis.
Set $\Omega = \Gamma/\PP$.
Let $\Omega_1, \dots, \Omega_\ell$ be the connected components of $\Omega$.
For $i \in \{1, \dots, \ell \}$ we denote by $\PP_i$ the set of vertices of $\Omega_i$ and we set $Y_i = \cup_{X \in \PP_i} X$ and $\Gamma_i = \Gamma_{Y_i}$.
The set $\PP_i$ is an admissible partition of $Y_i$ and $\Gamma_i/\PP_i = \Omega_i$ is a tree for all $i \in \{1, \dots, \ell \}$.
Moreover, we have $A = A_{Y_1} * \cdots * A_{Y_\ell}$, hence $A$ is residually finite if and only if $A_{Y_i}$ is residually finite for all $i \in \{1, \dots, \ell\}$.
So, we can assume that $\Omega = \Gamma/\PP$ is a tree. 

Since $|\PP| \ge 3$, $\Omega$ has a vertex $X$ of valence $\ge 2$.
Choose $Y \in \PP$ connected to $X$ by an edge of $\Omega$.
Let $s \in X$ and $t \in Y$ such that $s$ and $t$ are connected by an edge of $\Gamma$.
Recall that by definition $s$ and $t$ are unique.
Let $Q'$ be the connected component of $\Omega_{\PP \setminus \{ X\}}$ containing $Y$, let $\PP_Q'$ be the set of vertices of $Q'$, let $U' = \cup_{Z \in \PP_Q'} Z$, let $U=U' \cup \{s\}$, and let $\PP_Q = \PP_Q' \cup \{ \{ s\} \}$.
Observe that $\PP_Q$ is an admissible partition of $U$, that $A_Z$ is residually finite for all $Z \in \PP_Q$, that $\Gamma_U/\PP_Q$ is a tree, and that $|\PP_Q| < |\PP|$.
By the inductive hypothesis it follows that $A_U$ is residually finite.
Let $R$ be the connected component of $\Omega_{\PP \setminus \{ Y \}}$ containing $X$, let $\PP_R$ be the set of vertices of $R$, and let $V = \cup_{Z \in \PP_R} Z$.
Observe that $\PP_R$ is an admissible partition of $V$, that  $A_Z$ is residually finite for all $Z \in \PP_R$, that $\Gamma_V/\PP_R$ is a tree, and that $|\PP_R| < |\PP|$.
By the inductive hypothesis it follows that $A_V$ is residually finite.
Let $\Delta_1, \dots, \Delta_q$ be the connected components of $\Gamma_{S \setminus \{s\}}$.
Let $i \in \{1, \dots, q\}$.
Let $Z_i$ be the set of vertices of $\Delta_i$.
It is easily seen that either $Z_i \cup \{s \} \subset U$, or $Z_i \cup \{s\} \subset V$, hence, by the above, $A_{Z_i \cup \{ s\}}$ is residually finite. 
We conclude by Lemma \ref{lem2_3} that $A$ is residually finite.
\end{proof}



\end{document}